\documentclass[11pt, a4paper]{amsart}
\usepackage[hmargin=3cm, vmargin=2.5cm, includefoot, twoside]{geometry}
\usepackage[backref=page,
bookmarksopen=true]{hyperref}
\usepackage[latin1]{inputenc}  
\usepackage[T1]{fontenc}       
\usepackage[english]{babel}
\usepackage{amsmath}
\usepackage{amssymb}
\usepackage{amsthm}
\usepackage{latexsym}
\usepackage{mathrsfs}
\usepackage{graphics, graphicx}
\newtheorem{prop}{Proposition}[section]
\newtheorem{thm}[prop]{Theorem}
\newtheorem{conj}[prop]{Conjecture}

\newtheorem*{addendum*}{Addendum}
\newtheorem{cor}[prop]{Corollary}
\newtheorem{lem}[prop]{Lemma}

\newtheorem*{convention*}{Convention}
\theoremstyle{definition}
\newtheorem*{defn*}{Definition}
\newtheorem{defn}[prop]{Definition}
\newtheorem*{remark*}{Remark}
\newtheorem*{remarks*}{Remarks}
\newtheorem{remark}[prop]{Remark}
\theoremstyle{remark}

\newtheorem*{example*}{Example}
\numberwithin{equation}{section}

\newcommand{\vareps}{\varepsilon}

\newcommand{\NN}{\mathbf{N}}
\newcommand{\RR}{\mathbf{R}}
\newcommand{\ZZ}{\mathbf{Z}}

\newcommand{\QH}{\widetilde{\mathrm{QH}}}

\newcommand{\la}{\langle}
\newcommand{\ra}{\rangle}
\newcommand{\inv}{^{-1}}

\newcommand{\minus}{\setminus}%

\newcommand{\centra}{\mathscr{Z}}

\def\bs#1.{
              \def\temp{#1}
              \ifx\temp\empty
                   \mathcal{B}
              \else
                   \mathcal{B}(#1)
              \fi
}
\newcommand{\cat}{{\rm CAT(0)}\ }  
\newcommand{\tangle}[2]
{\angle_\mathrm{T}(#1,#2)}
\newcommand{\aangle}[3]
{\angle_{#1}(#2,#3)}
\newcommand{\cangle}[3]
{\overline{\angle}_{#1}(#2,#3)}
\DeclareMathOperator{\proj}{proj} 
\DeclareMathOperator{\Stab}{Stab}  \DeclareMathOperator{\Res}{Res}
 
\DeclareMathOperator{\Hom}{Hom} 
\DeclareMathOperator{\Pc}{Pc}
\DeclareMathOperator{\Isom}{Is}

\newcommand{\bd}{\partial} %

\def\Aut{\mathop{\mathrm{Aut}}\nolimits}

\def\min{\mathop{\mathrm{min}}\nolimits}

\def\cl{\mathop{\mathrm{cl}}\nolimits}
\def\scl{\mathop{\mathrm{scl}}\nolimits}
\begin{document}
\title{Rank one isometries of buildings and quasi-morphisms of Kac--Moody groups}
\author{Pierre-Emmanuel Caprace*}
\address{Université catholique de Louvain, Chemin du Cyclotron 2, 1348 Louvain-la-Neuve, Belgium}
\email{pe.caprace@uclouvain.be}
\thanks{*F.N.R.S. research associate}
\author{Koji Fujiwara$^\ddagger$}
\address{Graduate School of Information Science, Tohoku University,
Sendai, 980-8579, Japan}
\email{fujiwara@math.is.tohoku.ac.jp }
\thanks{$^\ddagger$Supported in part by
Grant-in-Aid for Scientific Research (No. 19340013).
}
\date{31 August 2008}
\begin{abstract}
Given an irreducible non-spherical non-affine (possibly non-proper) building $X$, we give sufficient conditions
for a group $G < \Aut(X)$ to admit an infinite-dimensional space of non-trivial quasi-morphisms. The result
applies in particular to all irreducible (non-spherical and non-affine) Kac--Moody groups over integral domains.
In particular, we obtain finitely presented simple groups of infinite commutator width, thereby answering a
question of Valerii G. Bardakov \cite[Problem~14.13]{Kourovka}. Independently of these considerations, we also
include a discussion of rank one isometries of \emph{proper} \cat spaces from a rigidity viewpoint. In an
appendix, we show that any homogeneous quasi-morphism of a locally compact group with integer values is
continuous.
\end{abstract}
\maketitle
\section{Introduction}

Let $G$ be a group. Recall that a \textbf{quasi-morphism} is a map $f: G \to \RR$ such that
$$\sup_{g, h \in G} |f(gh) -f(g) - f(h) | < \infty.$$
A quasi-morphism is called \textbf{homogeneous} if its restriction to every cyclic subgroup is a homomorphism.
The set $\mathrm{QH}(G)$ of all quasi-morphisms is naturally endowed with the structure of a real vector space.
We denote by $\QH(G)$ the vector space of \textbf{non-trivial quasi-morphisms}, namely
$$
\QH(G) = \mathrm{QH}(G)/\big(\ell^\infty(G) \oplus \Hom(G, \RR)\big).
$$
The space $\QH(G)$ naturally identifies to the kernel of the canonical map
$$H^2_{\mathrm{b}}(G, \RR) \to H^2(G, \RR)$$
of the second bounded cohomology space with trivial coefficients to ordinary second cohomology. Groups $G$ with
vanishing $\QH(G)$ include all amenable groups and all irreducible lattices in higher rank semisimple algebraic
groups over local fields \cite{BurgerMonod}. Opposite to these are groups with infinite-dimensional space of
quasi-morphisms; they include non-elementary hyperbolic groups \cite{EpsteinFujiwara}, mapping class groups of
surfaces of higher genus \cite{BestvinaFujiwara:MCG} and outer automorphism groups of free groups
\cite{Hamenstadt}, \cite{BestvinaFeighn}. There exist groups $G$ which have Kazhdan's property (T)
such that $\QH(G)$ is finite-dimensional but non-zero
\cite{MonodRemy}.

\medskip Before stating the main result of this paper, let us recall that a building of type $(W,S)$ is a set $X$
endowed with a map $\delta : X \times X \to W$ satisfying three simple axioms which are recalled in
Sect.~\ref{sect:buildings} below. The map $\delta$ is called the \textbf{Weyl distance}. A group $\Gamma$ acting
on $X$ by automorphisms is said to be \textbf{Weyl-transitive} if for all $x, y, x', y' \in X$ with $\delta(x,
y) = \delta(x', y')$ there exists $\gamma \in \Gamma$ such that $\gamma.x = x'$ and $\gamma.y = y'$.

\begin{thm}\label{thm:main}
Let $(W,S)$ be an irreducible non-spherical and non-affine Coxeter system with $S$ finite, $X$ be a building of
type $(W, S)$ and $G$ be a group acting on $X$ by automorphisms. Assume that at least one of the following
conditions is satisfied:
\begin{enumerate}
\item The $G$-action on $X$ is Weyl-transitive.

\item For some apartment $A \subset X$, the stabilizer $\Stab_G(A)$ acts cocompactly on $A$.
\end{enumerate}
Then $\QH(G)$ is infinite-dimensional.
\end{thm}

\begin{remark}\label{rem:intro}\
\renewcommand{\theenumi}{\alph{enumi}}
\begin{enumerate}
\item Notice that the building $X$ is not assumed to be locally compact. Moreover $X$ may contain flats of
arbitrarily large dimension; in particular it need not be Gromov hyperbolic.

\item The quasi-morphisms appearing in Theorem~\ref{thm:main} take values in $\ZZ$ and extend to quasi-morphisms
$\Aut(X) \to \ZZ$ defined over the full automorphism group of $X$. It is an amusing fact that \emph{any
homogeneous quasi-morphisms of a locally compact group with values in $\ZZ$ is continuous}; this will be shown
in the appendix below. In the special case when $X$ is locally compact, the bounded-open topology gives
$\Aut(X)$ the structure of a locally compact (second countable) group. In particular, under the assumptions of
Theorem~\ref{thm:main}, we deduce that the space of \emph{continuous} non-trivial quasi-morphisms on $\Aut(X)$
is infinite-dimensional.

\item Condition (1) in Theorem~\ref{thm:main} implies in particular that $G$ is far from discrete. We point out
that, in the special case when $X$ is Gromov hyperbolic,
a related transitivity assumption (which is not logically correlated to Weyl transitivity, but which might look qualitatively stronger at a first sight)
would automatically imply vanishing of $\QH(G)$. Indeed, according to an unpublished result by N.~Monod, independently
established by Ursula Hamenst\"adt \cite[Th.~4.1]{Hamenstadt:CrossRatio}, if a group $\Gamma$ admits a
quasi-distance-transitive action on a Gromov-hyperbolic geodesic metric space, then $\QH(\Gamma) = 0$ (for a
related result see Corollary~\ref{cor:Hamenstaedt:Prop6.4} below). By a \textbf{quasi-distance-transitive}
action, we mean that there exists $C
>0$ such that for all $x, y, x', y' \in X$ with $d(x, y) = d(x', y')$, there exists $\gamma \in \Gamma$ such
that $d(\gamma.x, x') \leq C$ and $d(\gamma.y, y') \leq C$.
\end{enumerate}

\end{remark}

\medskip
The most important class of groups admitting Weyl-transitive actions on buildings of arbitrary type is provided
by Kac--Moody groups. These groups are obtained by a functorial construction which associates a group functor on
the category of commutative rings to any generalized Cartan matrix (or more generally to any Kac--Moody root
datum), see \cite{TitsJA} and \cite{TitsLMS} for the split case and \cite{RemAst} for the almost split case.

\begin{cor}\label{cor:KM:CommutatorWidth}
Let $\mathcal{G}$ be a Kac--Moody-Tits functor whose Weyl group is irreducible non-spherical and non-affine and
let $R$ be an integral domain. Then $\QH(\mathcal{G}(R))$ is infinite-dimensional. In particular
$\mathcal{G}(R)$ possesses elements of strictly positive stable commutator length and is therefore of infinite
commutator width.
\end{cor}

Recall that the \textbf{stable commutator length} of an element $g$,
denoted by $\scl(g)$,
 of the commutator subgroup $[G, G]$ of a
group $G$ is defined as $\lim_{n \to \infty} \frac{\cl(g^n)}{n}$, where $\cl(h)$
denotes the minimal number $k$
such that $h \in [G, G]$ may be written as a product of $k$ commutators. 
The \textbf{commutator width} of $G$ is the supremum of the function $\cl$ on $[G,G]$.
The connection between the second
bounded cohomology of $G$, quasi-morphisms on $G$ and the stable commutator length of elements in $[G, G]$ was discovered by
Ch.~Bavard~\cite{Bavard}. 
In particular, if $f(g)>0$ for $f \in \QH(G)$ and $g \in [G,G]$,
then $\scl(g)>0$.
We refer to a recent monograph by D. Calegari \cite{Cal} for more information on the stable commutator length.

\medskip
A Kac--Moody group over a field of cardinality~$\geq 4$ is perfect. However even over the smallest fields, the
abelianization is always finite (this follows from the last two paragraphs of the proof of Theorem~15 in
\cite{CR06}). The arguments of \emph{loc.~cit.} show that the group $\mathcal{G}(R)$ as in
Corollary~\ref{cor:KM:CommutatorWidth} generally admits no nontrivial finite quotient. In fact, when $R$ is a
finite field of order larger than the rank of the Weyl group, the group $\mathcal{G}(R)$ happens to be simple,
and even finitely presented when the Weyl group is $2$-spherical \cite{CR06}. In particular, we obtain the
following, which answers positively a question asked by Valerii~G.~Bardakov \cite[Problem~14.13]{Kourovka}:

\begin{cor}\label{cor:KM:simple}
There exists an infinite family of pairwise non-isomorphic finitely presented simple groups possessing elements
of strictly positive stable commutator length; these groups have therefore infinite commutator width.
\end{cor}

Formerly finitely generated simple groups of infinite commutator width had been constructed by Alexey
Muranov~\cite{Muranov}.

\smallskip
As discussed in \cite{CR06}, properties of Kac--Moody groups over finite fields may be fruitfully compared to
properties of higher rank arithmetic groups; this comparison highlights strong analogies between both families.
Corollary~\ref{cor:KM:CommutatorWidth} testifies for the fact that Kac--Moody groups also enjoy some form of
hyperbolicity property, as opposed to the higher rank lattices. In order to stress this in a slightly different
way, we include the following corollary, which follows immediately from Corollary~\ref{cor:KM:CommutatorWidth}:

\begin{cor}\label{cor:BddGeneration}
Kac--Moody groups as in Corollary~\ref{cor:KM:CommutatorWidth} do not have bounded generation. Moreover they are
not boundedly generated by any family of torsion amenable subgroups.\qed
\end{cor}

The proof of Theorem~\ref{thm:main} relies on a construction of quasi-morphisms for groups acting on \cat
spaces, elaborated by M.~Bestvina and K.~Fujiwara \cite{BestFuji_symsp}. The conditions ensuring an
infinite-dimensional space of quasi-morphisms are recalled in Sect.~\ref{sec:prel}; the main one is the
existence of \emph{contracting isometries}, which by definition are isometries inducing a North-South dynamics
on the boundary and generalize the \emph{rank one isometries} as defined by W.~Ballmann. The key geometric
ingredients for the proof of Theorem~\ref{thm:main} are a charaterization of contracting isometries
(Theorem~\ref{thm:CharRankOne}) and a criterion ensuring that two given contracting isometries are
\emph{independent} and \emph{non-equivalent}. In fact, we believe that the hypotheses of Theorem~\ref{thm:main}
are unnecessarily strong for the existence of contracting isometries of buildings. To be more precise we propose
the following:

\begin{conj}\label{conj}
Let $(W,S)$ be an irreducible non-spherical and non-affine Coxeter system with $S$ finite, $X$ be a building of
type $(W, S)$ and $G$ be a group acting on $X$ by automorphisms without fixing any point at infinity (in the
\cat realization of $X$). Then  $G$ either stabilizes a proper residue or contains a contracting isometry.
\end{conj}

This conjecture holds in the special case where $W$ is Gromov hyperbolic, or more generally when $W$ is
relatively hyperbolic with respect to its maximal virtually Abelian subgroups. The latter condition is
completely characterized in~\cite{Ca07} and therefore yields the following:

\begin{prop}\label{prop:conj}
Assume that for any two infinite special subgroups $W_{J_1}, W_{J_2}< W$ such that $[W_{J_1}, W_{J_2}]=1$, the
group $\la W_{J_1} \cup W_{J_2}\ra$ is virtually abelian. Then any locally finite building of type $(W, S)$
satisfies Conjecture~\ref{conj}.
\end{prop}

A major interest of a solution to Conjecture~\ref{conj} is that, when combined with the Burger--Monod vanishing
theorem~\cite[Theorems~20 and~21]{BurgerMonod} and the Bestvina--Fujiwara construction, it would yield an
interesting rigidity statement for higher rank lattices, in the same vein as those established
in~\cite{BestvinaFujiwara:MCG} and~\cite{Hamenstadt}. In order to illustrate this, we mention the following
result, which should be compared to~\cite[Theorem~2]{Hamenstadt_IsomHyperbolic}.

\begin{thm}\label{th:proper}
Let $X$ be a proper \cat space and $G < \Isom(X)$ be any group of isometries. Assume that $G$ contains a rank
one element.
Let $\overline G$ be the closure of $G$ in $\Isom(X)$ with the compact-open
topology.
 Then one of the following assertions holds, where $\Lambda$ denotes the limit set of $G$:
\begin{itemize}
\item[(1)] 
$G$ either fixes a point in $\bd X$ or stabilizes a geodesic line; in both cases, it possesses a subgroup of
index at most~$2$ with infinite Abelianization.

\item[(2)] $\overline G$ acts transitively on $\Lambda \times \Lambda - \Delta$, where $\Delta$ denotes the
diagonal. 

\item[(3)] $\overline G$ does not act transitively on $\Lambda \times \Lambda - \Delta$, and the spaces $\QH(G)$
and $\QH_c(\overline G)$ are both infinite-dimensional.
\end{itemize}
Furthermore, if $X$ has cocompact isometry group, then $(1)$ implies that $\overline G$ is amenable and $(2)$
implies that the space of continuous nontrivial quasi-morphisms $\QH_c(\overline G)$ vanishes.
\end{thm}

Theorem~\ref{th:proper} has the following consequence, which directly relates to Conjecture~\ref{conj}:

\begin{cor}\label{cor:HigherRank}
Let $\Gamma < G = \prod_{\alpha \in A}\mathbf{G}_\alpha(k_\alpha)$ be an irreducible lattice, where $|A|>1$,
$(k_\alpha)_{\alpha \in A}$ is a finite family of local fields and the $\mathbf{G}_\alpha$ are connected simply
connected $k_\alpha$-almost simple groups of $k_\alpha$-rank $>1$. Let $X$ be a proper \cat space and $\varphi :
\Gamma \to \Isom(X)$ be any homomorphism. Then $\varphi(\Gamma)$ does not contain any rank one element.
\end{cor}

In particular, combining Proposition~\ref{prop:conj} with Corollary~\ref{cor:HigherRank}, one obtains:

\begin{cor}\label{cor:KM:RelHyp}
Let $\Gamma$ be as in Corollary~\ref{cor:HigherRank} and $X$ be a locally finite building whose type satisfies
the condition of Proposition~\ref{prop:conj}. Then any $\Gamma$-action on $X$ by automorphisms stabilises a
residue of spherical or Euclidean type.\qed
\end{cor}

Let us finally mention that, independently of M.~Bestvina and K.~Fujiwara, Ursula Hamenst\"adt developed a
slightly different approach providing a general axiomatic setting for groups acting on topological spaces by
homeomorphisms to admit an infinite-dimensional space of non-trivial quasi-morphisms \cite{Hamenstadt}. It turns
out that her approach, when applied to the present context, would also provide information on the second bounded
cohomology with nontrivial coefficients.

\medskip
The paper is organized as follows. Section~\ref{sec:prel} is preliminary. The aim of Section~\ref{sec:rigidity}
is the proof of Theorem~\ref{th:proper} and its corollaries. It contains several geometrical results on rank one
isometries of proper \cat spaces which might be of some independent interest. Sect.~\ref{sec:Coxeter} is devoted
to Coxeter groups; its main purpose is to show that irreducible Coxeter groups which are not virtually abelian
contain many contracting isometries in their natural action on the Davis complex. Finally, hyperbolic isometries
of buildings are studied in Sect.~\ref{sect:buildings}.

\bigskip \noindent \textbf{Convention.}
In order to avoid any confusion, we remark that quasi-morphisms
are sometimes called ``quasi-homomorphisms'' in the literature (for example
in \cite{BestvinaFujiwara:MCG},\cite{BestFuji_symsp}).
This explains the notation  $\mathrm{QH}$ 
for the space of quasi-morphisms;
and $\QH$ for the non-trivial ones in this paper, 
following \cite{BestvinaFujiwara:MCG},\cite{BestFuji_symsp}.
We note however that in the monograph \cite{Cal}, the space of quasi-morphisms is
denoted by $\hat Q$ and the space of homogeneous quasi-morphisms
is denoted by $Q$, while the latter is denoted by  $\mathrm{HQH}$ (for `\emph{homogeneous quasi-homomorphisms}') in \cite{BestvinaFujiwara:MCG},\cite{BestFuji_symsp}. 
In the present paper, no special notation is used for the space of homogeneous quasi-morphisms.

Stable commutator length of $g$ is denoted by $\mathrm{scl}(g)$
in this paper, but it is sometimes called `\emph{stable length}',
and also denoted by $\|g\|$ (as for example in \cite{Bavard}).

\bigskip \noindent \textbf{Acknowledgements.} Both authors thank the MSRI, Berkeley and particularly the organizers of
the special program on geometric group theory which was held there
in the Fall 2007 and during which this work was initiated. The first
author also acknowledges support from the European Post-Doctoral
Institute (EPDI). He is grateful to Ursula Hamenst\"adt for
suggesting that some of her results on bounded cohomology would also
provide relevant information in the context of the present paper.
Section~\ref{sec:rigidity} below was inspired by this conversation.
The second author would like to thank Mladen Bestvina for intensive
discussions on Coxeter groups. We thank Nicolas Monod for
stimulating conversations and for numerous useful comments and
suggestions on a preliminary version of this paper. 
We are grateful to Danny Calegari for his interest and useful suggestions.
Finally, we
thank the referee for his/her comments.
\section{Rank one elements, contracting isometries and quasi-morphisms}\label{sec:prel}

Let $X$ be a \cat space. A geodesic line $L$ in $X$ is said to have \textbf{rank one} if it does not bound a
flat half-plane. The line $L$ is said to be \textbf{$B$-contracting} for some $B \geq 0$ if for every metric
ball $C$ disjoint from $L$ the projection $\pi_L(C)$ has diameter at most $B$.

An isometry $\gamma \in \Isom(X)$ is said to have \textbf{rank one} if it is hyperbolic and if some (and hence
any) of its axes has rank one. Similarly $\gamma$ is called \textbf{$B$-contracting} if it is hyperbolic and if
some of its axes is $B$-contracting. It is called \textbf{contracting} if it is $B$-contracting for some $B \geq
0$.

Recall from \cite[Thm.~5.4]{BestFuji_symsp} that if $X$ is proper, then an isometry has rank one if and only if
it is $B$-contracting for some $B \geq 0$. We will see later (see Theorem~\ref{thm:CharRankOne}) that for some
class of finite-dimensional \cat spaces, this assertion holds even without the properness assumption.

Following \cite{BestFuji_symsp}, we will use:
\begin{defn}
Let $\gamma_1, \gamma_2 \in \Gamma$ be hyperbolic elements and fix a base point $x_0 \in X$.

The elements $\gamma_1$ and $\gamma_2$ are called \textbf{independent} if the map
$$\ZZ \times \ZZ \to [0, \infty) : (m, n) \mapsto d(\gamma_1^m.x_0, \gamma_2^n.x_0)$$
is proper.

The elements $\gamma_1$ and $\gamma_2$ are called \textbf{$\Gamma$-equivalent} (notation: $\gamma_1 \sim_\Gamma
\gamma_2$) if the following condition holds: there exists $\delta >0$ such that for each $r > 0$ there is $g \in
\Gamma$ with $d(\gamma_1^n.x_0, g \gamma_2^n.x_0) < \delta$ for all integers $n \in [-r, r]$.
\end{defn}

Notice that both properties are independent of the choice of the base point. Furthermore two elements $\gamma_1$
and $\gamma_2$ are independent if and only if $\gamma_1$ and $\gamma_2\inv$ are independent. When the
$\Gamma$-action is proper, both notions  can be made more precise:

\begin{lem}\label{lem:irreversible}
Let $\Gamma$ act properly discontinuously on a complete \cat space $X$. Then:
\begin{itemize}
\item[(i)] Two hyperbolic elements $\gamma_1, \gamma_2$ are independent if and only if the canonical attracting
fixed point $\gamma_1^+$ of $\gamma_1$ at infinity is distinct from both the attracting and the repulsive fixed
point of $\gamma_2$ at infinity.

\item[(ii)] Two contracting elements $\gamma_1, \gamma_2$ satisfy $\gamma_1 \sim_\Gamma \gamma_2$ if and only if
some positive powers of $\gamma_1$ and $\gamma_2$ are conjugate.

\item[(iii)] If two contracting elements $\gamma_1$ and $\gamma_2$ are not independent, then $\gamma_1
\sim_\Gamma \gamma_2$ or $\gamma_1 \sim_\Gamma \gamma_2\inv$.

\item[(iv)] If $\Gamma$ is non-elementary and contains a contracting element, then it contains two contracting
elements $\gamma_1$ and $\gamma_2$ such that $\gamma_1 \not \sim_\Gamma \gamma_2$ and $\gamma_1 \not \sim_\Gamma
\gamma_2\inv$.
\end{itemize}
\end{lem}

Assertion (i) means in other words that given two rank one elements, either their axes are parallel (and the
element are dependent) or the respective attracting and repulsive fixed points of the two elements at infinity
form four distinct points of the visual boundary $\bd X$.

\begin{proof}
(i). The `only if' part is clear. Let $\gamma_1$ and $\gamma_2$ be hyperbolic elements with non-parallel axes,
say  $\ell_1$ and $\ell_2$ respectively, and assume for a contradiction that $\gamma_1$ and $\gamma_2$ are not
independent. Then $\ell_1$ and $\ell_2$ contain rays $\rho_1 \subset \ell_1$ and $\rho_2 \subset \ell_2$ which
are at finite Hausdorff distance from one another. By properness of the $\Gamma$-action, it follows that there
exist integers $m \neq m'$ and $n \neq n'$ such that $g_1^m g_2^n = g_1^{m'} g_2^{n'}$. Thus $g_1^{m''} =
g_2^{n''}$ for some nonzero $m''$ and $n''$. This implies that $\ell_1$ and $\ell_2$ are at finite Hausdorff
distance from one another, hence parallel. This is absurd.

\medskip \noindent (ii). %
See \cite[Prop.~6.5(3)]{BestFuji_symsp}.

\medskip \noindent (iii). %
Follows from (i), (ii) and the fact that the stabilizer the parallel set of any axis is virtually cyclic by
properness.

\medskip \noindent (iv). %
Follows from \cite[Prop.~6.2 and~6.5(4)]{BestFuji_symsp}.
\end{proof}

\begin{remark}\label{rem:conjugacy}
Important to observe is that, when the action is discrete as above, \emph{the property of being
$\Gamma$-inequivalent is conjugacy--invariant}. More precisely, given $\gamma_1, \gamma_2 \in \Gamma$ with
$\gamma_1 \not \sim_\Gamma \gamma_2$ and $\gamma_1 \not \sim_\Gamma \gamma_2\inv$ as in (iv), then, for any $g
\in \Gamma$, we have $\gamma_1 \not \sim_\Gamma g\gamma_2 g\inv$ and furthermore $\gamma_1$ and $g\gamma_2g\inv$
are independent. This follows from (ii) and (iii) in Lemma~\ref{lem:irreversible}.
\end{remark}

The construction of quasi-morphisms that we will use was performed by M.~Bestvina and K.~Fujiwara
\cite[Th.~6.3]{BestFuji_symsp}; notice that there is no discreteness assumption whatsoever on the action:

\begin{prop}\label{prop:infiniteHQH}
Let $\Gamma < \Isom(X)$ be any group of isometries of a complete \cat space $X$. Assume that $\Gamma$ contains
two independent elements which are not $\Gamma$-equivalent. Then $\QH(\Gamma)$ is infinite-dimensional.\qed
\end{prop}

\section{Proper \cat spaces with rank one isometries and rigidity of higher rank lattices}\label{sec:rigidity}

The present section is aimed at proving Theorem~\ref{th:proper}; there is no logical dependence between this and
the subsequent sections.

Given a proper \cat space $X$, the compact-open topology gives  $\Isom(X)$ the structure of a locally compact
second countable topological group.

\subsection{The stabilizer of a point at infinity fixed by a rank one element}

\begin{prop}\label{prop:amenable:fixator}
Let $X$ be a proper \cat space with cocompact isometry group and $G < \Isom(X)$ be any group of isometries. Let
$\xi \in \bd X$ be a point fixed by some rank one element of $\Isom(X)$. Then the stabilizer $\overline G_\xi$
is amenable.
\end{prop}
\begin{proof}
By the very nature of the statement to be established, there is no loss of generality in assuming that $G$ is closed.
If $G$ stabilizes a geodesic line, then the desired conclusion clearly holds. We assume henceforth that $G$ does
not stabilize any line. In particular, the existence of a rank one element implies that $G$ does have a global
fixed point at infinity. Thus $X$ possesses a nonempty minimal $G$-invariant closed convex subset $Y \subseteq
X$ (see \cite[Prop~3.1]{CaMo}). The point-wise stabilizer of $Y$ being compact, hence amenable, there is no loss
of generality in assuming that $Y=X$.

By \cite[Th.~4.7]{CaMo} the group $G$ is either an almost connected simple Lie group or totally disconnected. In
the former case, the desired result follows from \cite[Th.~6.4]{CaMo}. In the latter case, the result follows
from \cite[Th.~1.5]{Caprace:amenable} since for any fixed point $\xi$ of a rank one isometry, the transversal
space $X_\xi$ as defined in~\emph{loc.~cit.} is bounded.
\end{proof}

The following statement parallels Proposition~6.4 in \cite{Hamenstadt_IsomHyperbolic}:

\begin{cor}\label{cor:Hamenstaedt:Prop6.4}
Let $X$ be a proper \cat space with cocompact isometry group and $G < \Isom(X)$ be a closed group of isometries
with limit set $\Lambda$. Assume that $G$ contains a rank one element. If $ G$ acts transitively on $\Lambda
\times \Lambda - \Delta$, then the space of continuous nontrivial quasi-morphisms $\QH_c( G)$ vanishes.
\end{cor}

\begin{proof}
Let $g \in G$ be the given rank one element and $a, b \in \bd X$ denote its fixed points. Let also $G_{\{a,
b\}}$ denote the stabilizer of the pair $\{a, b\}$ in $G$. By assumption, for each $h \in G$ there exists $g'
\in G$ such that $g'.h \in G_b$, and we may choose $g' \in G_a$ or $g' \in G_{\{a, b\}}$ according as $h(b) \neq
a$ or $h(b) \neq a$. This shows that the group $G$ is a product
$$G = G_a \cdot G_{\{a, b\}} \cdot G_b.$$
By Proposition~\ref{prop:amenable:fixator} all subgroups  $G_a$, $G_b$  and $G_{\{a, b\}}$ are amenable. In
particular the spaces  $ H^2_{\mathrm{cb}}(G_a, \RR)$, $  H^2_{\mathrm{cb}}(G_b, \RR)$ and
$H^2_{\mathrm{cb}}(G_{\{a, b\}}, \RR)$ vanish and any continuous nontrivial quasi-morphism of $G_a$, $G_b$ or
$G_{\{a, b\}}$ is bounded. Thus the same holds for $G$ as desired.
\end{proof}

\begin{remark}\label{rem:cell}
If $X$ is a \cat cell complex with finitely many types of cells and $G < \Isom(X)$ acts by cellular
transformations, then Proposition~\ref{prop:amenable:fixator} and Corollary~\ref{cor:Hamenstaedt:Prop6.4} remain
true without the hypothesis that $X$ has cocompact isometry group: this follows from the same reasoning as
above, using a straightforward adaptation of the arguments in~\cite{Caprace:amenable}.
\end{remark}

\subsection{On the existence of independent rank one elements}

\begin{prop}\label{prop:amenable:independent}
Let $X$ be a proper \cat space and $G < \Isom(X)$ be any subgroup. Assume that $G$ contains a rank one element.
Then one of the following assertions holds:
\begin{itemize}
\item[(1)] $G$ either fixes a point in $\bd X$ or stabilize a geodesic line; in both cases, it possesses a
subgroup of index at most~$2$ with infinite Abelianization. Furthermore, if $X$ has cocompact isometry group,
then $\overline G < \Isom(X)$ is amenable.

\item[(2)] $G$ contains two independent rank one elements; in particular $\overline G$ contains a discrete
non-Abelian free subgroup.
\end{itemize}
\end{prop}

\begin{proof}
Any rank one element acts on the boundary at infinity with a North-South dynamics, see
\cite[Lem.~3.3.3]{BallmannLN}. Therefore, if $\overline G$ contains two independent rank one elements then the
existence of a discrete non-Abelian free subgroup follows from a standard ping-pong argument. We assume
henceforth that $G$ does not contain any pair of independent rank one elements. In particular any two rank one
elements of $G$ have a common fixed point in $\bd X$.

We claim that every triple of rank one elements of $G$ have a common fixed point at infinity. Otherwise there
would exist three rank one elements $g_1, g_2, g_3$ such that $g_1$ and $ g_2$ have a common attracting fixed
point, say $a$, and there respective repelling fixed points $b_1$ and $b_2$ are precisely the fixed points of
$g_3$. Conjugating $g_3$ by a large positive power of $g_1\inv$, we then obtain a rank one element which is
independent of $g_2$. This is a contradiction.

>From this claim, it follows that all rank one elements of $G$ have a common fixed point at infinity, say $\xi$.
In particular, the normal subgroup $N \lhd G$ generated by all rank one elements of $G$ fixes $\xi$. Since any
rank one element of $G$ has exactly two fixed points at infinity, it follows that $N$ has at most two fixed
points as well. In particular $G$ has a subgroup of index at most~$2$ which fixes $\xi$, and the Busemann
character centered at $\xi$ yields a homomorphism of this subgroup taking values in~$\RR$ (see \emph{e.g.}
\cite[\S4.3]{Caprace:amenable}).

Passing to the closure, we deduce that $\overline G$ has a closed subgroup of index~$2$ which fixes $\xi$. Thus
$\overline G$ is amenable as soon as $\Isom(X)$ is cocompact by Proposition~\ref{prop:amenable:fixator} and
Assertion~(1) holds.
\end{proof}

\begin{prop}\label{prop:LimitSet}
Let $X$ be a proper \cat space and $G < \Isom(X)$. Assume that $G$ contains two independent rank one elements.
Then the set of pairs of fixed points of rank one elements of $G$ is dense in $\Lambda \times \Lambda - \Delta$,
where $\Lambda$ denotes the limit set of $G$ and $\Delta \subset \Lambda \times \Lambda$ the diagonal.
\end{prop}

\begin{proof}
The proof goes along the same lines as that of \cite[Th.~4.1]{BallmannBrin}. For the reader's convenience, we
include the details of the argument.

\smallskip Following \emph{loc.~cit.}, we shall say that a pair of points $\xi, \eta \in \bd X$ is \textbf{dual}
(relative to $G$) if for all neighbourhoods $U$ and $V$ of $\xi$ and $\eta$ in the visual compactification
$\overline X$, there exists $g \in G$ such that
$$g(\overline X - U) \subset V \hspace{.5CM} \text{and} \hspace{.5CM} g\inv(\overline X - V) \subset U. $$
Notice that the set of points which are dual to some fixed $\xi \in \bd X$ is closed (with respect to the cone
topology).

The relevance of this notion comes from the following. The two fixed points of any given rank one element are
dual to each other: this follows from~\cite[Lem.~3.3.3]{BallmannLN}. Conversely, the results of
\cite[Sect.~III.3]{BallmannLN} imply that if $\{\xi, \eta\} \subset \bd X$ is a dual pair, then there exists a
sequence of rank one elements $g_n \in G$ such that the attracting and repelling fixed points of $g_n$ tend to
$\xi$ and $\eta$ as $n$ tends to infinity. All we need to show is thus that any point of $\bd X$ is dual to any
other.

Let $\gamma_1, \gamma_2 \in G$ be two independent rank one elements,  and $a_1, a_2$ (resp. $b_1, b_2$) denote
their respective attracting (resp. repelling) fixed points. By considering products of the form
$\gamma_1^m.\gamma_2^n$ for appropriately chosen integers $m,n$, one shows that any two distinct points in
$\{a_1, a_2, b_1, b_2\}$ are dual to one another.

Let now $\xi \in \Lambda - \{a_1, a_2, b_1, b_2\}$. Pick a base point $x_0 \in X$ and choose a sequence
$(g_n)_{n \geq 0}$ of elements of $G$ such that $\lim_n g_n.x_0 \to \xi$. By~\cite[Lem.~4.4]{BallmannBrin} we
have $\lim_n g_n.a_1 = \xi$ or $\lim_n g_n.b_1 = \xi$ (or both). Thus $g'_n = g_n \gamma_1 g_n\inv$ is a
sequence of rank one elements such that the corresponding sequence of attractive or repelling fixed points
converges to $\xi$. Upon multiplying each $g'_n$ by an appropriate power of $\gamma_2$ if necessary, we may
extract a subsequence $(g'_{n_k})$ such that $g'_{n_k}$ is independent from $\gamma_1$ for each $k\geq 0$ and
which still enjoys the same convergence property of its attracting or repelling fixed points. The preceding
paragraph then shows that $\xi$ is dual to both $a_1$ and $b_1$. From a symmetric argument we deduce that $\xi$
is also dual to both $a_2$ and $b_2$.

\smallskip
What we have done so far shows that for any four-tuple of points of $\Lambda$ which are pairwise dual, any other
point of $\Lambda$ is dual to each of them. In view of the hypothesis, this implies that any two points of
$\Lambda$ are dual.
\end{proof}

\subsection{Quasi-morphisms and rigidity}

\begin{lem}\label{lem:Hamenstaedt}
Let $X$ be a proper \cat space and $G < \Isom(X)$ be a closed subgroup with limit set $\Lambda$. For every
hyperbolic isometry $g \in G$ with attracting and repelling fixed points $a \neq b \in \Lambda$, the $G$-orbit
of $(a, b)$ is a closed subset of $\Lambda \times \Lambda$.
\end{lem}
\begin{proof}
Identical to the proof of \cite[Lem.~6.1]{Hamenstadt_IsomHyperbolic}.
\end{proof}

\begin{proof}[Proof of Theorem~\ref{th:proper}]
Assume that $G$ is \emph{non-elementary}, that is to say $G$ does not fix a point at infinity and does not
stabilize any geodesic line. By Proposition~\ref{prop:amenable:independent} it follows that $G$ contains two
independent rank one elements. Therefore the set of pairs of fixed points of rank one elements of $G$ is dense
in $\Lambda \times \Lambda - \Delta$.

Assume now that $\overline G$ does not act transitively on $\Lambda \times \Lambda - \Delta$. In view of
Lemma~\ref{lem:Hamenstaedt}, we deduce from the density assertion of the preceding paragraph that $G$ contains
two rank one elements $g_1, g_2$ with respective attracting and repelling fixed points $(a_1, b_1)$ and $(a_2,
b_2)$, such that the $\overline G$-orbit of $(a_1, b_1)$ is distinct from the $\overline G$-orbit of $(a_2,
b_2)$. It follows that $g_1$ and $g_2$ are $\overline G$-inequivalent.

Notice moreover that $b_1 \neq b_2$. Indeed, since $g^n_1.a_2$ tends to $a_1$ as $n$ tends to infinity, the
equality $b_1 = b_2$ would imply that $(a_2, b_2)$ is in the same $\overline G$-orbit as $(a_1, b_1)$, which is
absurd. Similarly one shows that the four points $a_1, a_2, b_1, b_2$ are pairwise distinct. In particular $g_1$
and $g_2$ are independent.

Therefore, we may apply Proposition~\ref{prop:infiniteHQH}, which shows that $\QH(G)$ and $\QH(\overline G)$ are
infinite-dimensional. Since the Bestvina--Fujiwara construction yields quasi-morphisms with integer values, it
follows from Theorem~\ref{thm:appendix} below that $\QH_c(\overline G)$ is infinite-dimensional as well.

\medskip
The last two assertions of Theorem~\ref{th:proper} follow from Proposition~\ref{prop:amenable:fixator} and
Corollary~\ref{cor:Hamenstaedt:Prop6.4} since the limit set of $G$ coincides with the limit set of its closure
$\overline G$.
\end{proof}

\begin{proof}[Proof of Corollary~\ref{cor:HigherRank}]
Since $\Gamma$ has Kazhdan's property~(T), every finite index subgroup of $\Gamma$ has finite Abelianization. On
the other hand, we know by \cite[Theorems~20 and~21]{BurgerMonod} that $\QH(\Gamma) = 0$.
Theorem~\ref{th:proper} thus implies that $H = \overline{\varphi(\Gamma)}$ is transitive on the pair of distinct
points of its limit set $\Lambda$.

Let $Y \subseteq X$ be a nonempty closed convex $H$-invariant subset; such a subspace exists since $H$ has no
fixed point at infinity in view of Proposition~\ref{prop:amenable:independent}. The version of Monod's
superrigidity theorem given in \cite[Theorem~9.4]{CaMo} (which may be applied since $\Gamma$ has property (T)
and is square-integrable \cite{Shalom2000}) provides a continuous homomorphism $\varphi:G \to \Isom(Y)$
extending the given $\varphi:\Gamma \to \Isom(Y)$.

Notice that since $Y$ admits a rank one isometry by assumption, it is irreducible. Since $\Gamma$ acts
minimally, it follows from \cite[Th.~1.6]{CaMo} that the continuous map $\varphi:G \to \Isom(Y)$ factors through
some simple factor of $G$, say $G_\alpha$.

Given a semisimple element $h \in G_\alpha$ which is not \emph{periodic} (\emph{i.e.} the cyclic subgroup $\la h
\ra$ is not relatively compact), the image $\varphi(h)$ is not an elliptic isometry, since any continuous
nontrivial homomorphism of a simple algebraic group to a locally compact second countable group is proper
\cite[Lemma~5.3]{BM96}. In particular the limit set of $\la \varphi(h) \ra$  in $Y$ is a nonempty subset of
$\Lambda$, thus consisting of~$1$ or~$2$ points.

Let now $P < G_\alpha$ be any proper parabolic subgroup. By \cite[Lemma~2.4]{Prasad77} there exists a
non-periodic semisimple element $h$ such that
$$ P = \centra_{G_{\alpha}}(h) \cdot U(h),$$
where $\centra_{G_{\alpha}}(h)$ is the centraliser of $h$ in
$G_{\alpha}$ and  $U(h) \lhd P$ is the \emph{contraction group} defined by
$$U(h) = \{g \in G_\alpha \; | \; \lim_{n \to \infty } h^n g h^{-n} = 1\}.$$
It follows that the limit point of the sequence $(h^{-n}.y_0)_{n \geq 0}$, where $y_0 \in Y$ is any base point,
is $P$-invariant. We have thus established that any proper parabolic subgroup of $G_\alpha$ fixes some point of
$\Lambda$. In particular the stabilizer of any $\xi \in \Lambda$ in $G_\alpha$ is a parabolic subgroup, since
any subgroup containing a parabolic is itself parabolic. In view of \cite[Lemma~4.4]{BallmannBrin} the pointwise
stabilizer of a triple of points of $\Lambda$ in $\Isom(Y)$ has a fixed point in $Y$ and is thus compact. Since
$\varphi: G_\alpha \to \Isom(Y)$ is proper, it follows that any parabolic subgroup of $G_\alpha$ has at most~$2$
fixed points in $\Lambda$. Thus any minimal parabolic subgroup of $G_\alpha$ is contained in at most two maximal
ones. This is absurd since $G_\alpha$ is simple and has $k_\alpha$-rank~$\geq 2$.
\end{proof}

\begin{remark}
It should be noted that the assumption that the group $G$ has at least two simple factors accounts for the
corresponding assumption in the version of the superrigidity theorem for higher rank lattices that we appeal to.
It is reasonable to conjecture the conclusion of Corollary~\ref{cor:HigherRank} still holds for lattices in
higher rank \emph{simple} groups; in fact, apart from the aforementioned superrigidity, all arguments of the
proof remain valid in that more general context.
\end{remark}


\section{Rank one elements in Coxeter groups}\label{sec:Coxeter}

Let $(W, S)$ be a Coxeter system such that the Coxeter group $W$ is finitely generated; equivalently $S$ is
finite. We denote by $\Sigma$ the associated Davis complex; it is endowed with proper \cat metric and a natural
properly discontinuous cocompact $W$-action by isometries
(See \cite{Davis}. This fact is proved by Moussong
\cite{Moussong}). We view the elements of $W$ as isometries of
$\Sigma$; thus by a \emph{hyperbolic element} of $W$ (resp. rank one, $B$-contracting, etc.) we wean an element
which acts on $\Sigma$ as a hyperbolic (resp. rank one, $B$-contracting, etc.) isometry.

\begin{lem}\label{lem:basic}
Let $\gamma \in W$ and $x, y \in \Sigma$ such that $\gamma.x = y$. Then $\gamma$ belongs to the group $W(x, y)$
generated by all those reflections which fix some point of the geodesic segment $[x, y]$.
\end{lem}
\begin{proof}
Let $C_x$ be a chamber containing $x$ and define $C_y = \gamma.C_x$. It is well known (and easy to see) that
$\gamma$ belongs to the group generated by all those reflections which fix a wall separating $C_x$ from $C_y$.
Clearly every wall separating $C_x$ from $C_y$ meets $[x, y]$; the result follows.
\end{proof}

\subsection{Parabolic closures and essential elements}

Recall that a subgroup of $W$ of the form $W_J$ for some $J \subset S$ is called a \textbf{standard parabolic
subgroup}. Any of its conjugates is called a \textbf{parabolic subgroup} of $W$. A basic fact on Coxeter groups
is that any intersection of parabolic subgroups is itself a parabolic subgroup
(see \cite{Krammer}). This allows one to define the
\textbf{parabolic closure} $\Pc(R)$ of a subset $R \subset W$: it is the smallest parabolic subgroup of $W$
containing $R$.

An element $w \in W$ is called \textbf{standard} if its parabolic closure is a standard parabolic subgroup. It
is called \textbf{cyclically reduced} if $\ell(w) = \min \{ \ell(\gamma w \gamma\inv) \; | \; \gamma \in W\}$.
The following elementary result shows in particular that any cyclically reduced element is standard:
\begin{prop}\label{prop:Pc}
Let $w \in W$. We have the following:
\begin{itemize}
\item[(i)] If $w$ is standard, then for any writing of $w$ as a reduced word $w = s_1 s_2 \cdots s_{\ell(w)}$
with letters in $S$, we have $\Pc(w) = \la s_1, s_2, \dots, s_{\ell(w)} \ra$.

\item[(ii)] Let $x \in W$ be such that $x \Pc(w) x\inv$ is standard. Assume moreover that
$$\ell(x) = \min\{\ell(\gamma) \; | \; \gamma \in W, \ \gamma \Pc(w)\gamma\inv \text{ is standard}\}.$$
If $\gamma \neq 1$, then $\ell(x w x\inv) < \ell(w)$.

\item[(iii)] Let $s \in S$. If $w$ is standard, then either $\Pc(ws) \subset \Pc(w)$ or $\Pc(ws) = \la \Pc(w)
\cup \{s \} \ra$ and $ws$ is standard as well.
\end{itemize}
\end{prop}
\begin{proof}
(i). This is well known; it is an immediate consequence of the solution to the word problem
\cite{Tits_CoxWordPb}.

\medskip \noindent (ii). %
Set $P=\Pc(w)$.
Let $J \subset S$ be such that $xPx\inv = W_J$ and set $R = x\inv W_J$. We view (the $1$-skeleton of) $\Sigma$
as a chamber system (see \cite{Weiss_book}); the chambers are the elements of $W$. The coset $R$ is the
$J$-residue of $\Sigma$ containing $x\inv$. Note that $\Stab_W(R) = P$. The condition that $x$ minimizes the
length of all elements which conjugate $P$ to a standard parabolic subgroup means precisely that $x$ is the
combinatorial projection of $1$ onto $R$ (see \cite[Th.~3.22]{Weiss_book} for the notion of projections onto
residues). Thus $x = \proj_R(1)$. Let also $y = proj_R(w)$. Since $R$ is $P$-invariant we have $w.x = y$. Thus
$d(x, y)=\ell(x w x\inv)$. Furthermore, in view of basic properties of the combinatorial projection
\cite[Th.~3.22]{Weiss_book}, every wall which separates $x$ from $y$ also separates $1$ from $w$. This implies
that $d(x, y) \leq d(1, w) = \ell(w)$. Therefore, if $\ell(x w x\inv) \geq \ell(w)$, then we deduce $\ell(x w
x\inv) = \ell(w)$ and the set $\mathscr{M}(x, y)$ of walls separating $x$ from $y$ coincides with the set
$\mathscr{M}(1, w)$. We have to show that $x = 1$.

Let $s_1 \cdots s_{\ell(w)}$ be a reduced word representing $w$. Notice that the reflections associated to walls
in $\mathscr{M}(1, w)$ are precisely
$$s_1, \; s_1 s_2 s_1, \; \dots \; , \; s_1 \cdots s_{\ell(w)} \cdots s_1$$
since $w = s_1 s_2 \dots s_{\ell(w)}$. By the above, each of these reflections stabilizes $R$ (see
\cite[Prop.~4.10]{Weiss_book}) and, hence, belongs to $P$. This shows that $P \supset \la s_1, s_2, \dots,
s_{\ell(w)} \ra$. Since the group generated by $ s_1, s_2, \dots, s_{\ell(w)}$ is clearly a parabolic subgroup
which contains $w$, we deduce $P = \Pc(w) = \la s_1, s_2, \dots, s_{\ell(w)} \ra$. In particular $P$ is standard
and hence $x$ must be trivial, as desired.

\medskip \noindent (iii). %
By assumption $w$ is standard. This implies that $Q := \la \Pc(w) \cup \{s\} \ra$ is a standard parabolic
subgroup.  Since $w$ and $s$ both belong to $Q$, it follows that $\Pc(ws) \subset Q$. If $s \in \Pc(w)$, then $Q
= \Pc(w)$ and, hence $\Pc(ws) \subset \Pc(w)$. We assume henceforth that $s \not \in \Pc(w)$. In particular
$\Pc(w)$ is properly contained in $Q$; more precisely the ranks of $\Pc(w)$ and $Q$ differ by $1$.

If $wsw\inv \in \Pc(ws)$, then $w = wsw\inv.ws$ belongs to $\Pc(ws)$ and hence $Pc(w) \subset \Pc(ws)$. Since
$s$ does not belong to $\Pc(w)$, we obtain
$$\Pc(w) \subsetneq \Pc(ws) \subset Q,$$
which implies that $\Pc(ws) = Q$ since these parabolic subgroups have the same rank. In particular $ws$ is
standard.

Assume now that $wsw\inv \not \in \Pc(ws)$. Set $P' = \Pc(ws)$ and choose a residue $R'$ whose stabilizer in $W$
is $P'$, in a similar way as in the proof of (ii). The condition that $wsw\inv$ does not stabilize $R'$ means
that the projections $\proj_{R'}(w)$ and $\proj_{R'}(ws)$ must coincide. Arguing as in the proof of (ii), we
deduce that every walls separating $\proj_{R'}(1)$ from $\proj_{R'}(ws)$ also separates $1$ from $w$. Since $w$
is standard, this implies in view of (i) that $\Pc(ws)$ is contained in $\Pc(w)$, as desired.
\end{proof}

An element $\gamma \in W$ is called \textbf{essential} if $\Pc(\gamma) = W$. The following result appears in
\cite[Thm.~3.4]{Paris_irrCox}; we give an alternative argument:

\begin{cor}\label{cor:EssentialElements}
Let $s_1, \dots, s_n$ be all the elements of $S$ (in any order). Then $w = s_1 \cdots s_n$ is essential.
\end{cor}
\begin{proof}
For each $k = 1, \dots, n$, let $w_k = s_1 \cdots s_k$. An immediate induction using Proposition~\ref{prop:Pc}
shows that $\Pc(w_k) = \la s_1, \dots, s_k \ra$.
\end{proof}

\subsection{Walls separating a flat half-space}

Euclidean flats in $\Sigma$ have been studied in \cite{CH06}. The following parallels some results from
\emph{loc. cit.} in the case of flat half-spaces:

\begin{prop}\label{prop:FlatHalfSpace}
Let $H$ be a flat half-space in $\Sigma$ and denote by $P$ the parabolic closure of the set of reflections
fixing some point of $H$. Then:
\begin{itemize}
\item[(i)] We have
$$P \cong K \times P_1 \times \dots \times P_k,$$
where each $P_i$ is an infinite parabolic subgroup and $K$ is a finite parabolic subgroup containing all
reflections which fix $H$ pointwise. Moreover, if none of the $P_i$'s is of affine type and rank~$\geq 3$, then
$k \geq \dim F$.

\item[(ii)] The parabolic subgroup $P$ contains every element $\gamma \in W$ which maps some point of $H$ into
$H$.
\end{itemize}
\end{prop}
\begin{proof}
Part (i) follows from a straightforward adaptation of the arguments from \cite{CH06} (see also
\cite[Prop.~3.1]{Ca07}).

Part (ii) follows from (i) and Lemma~\ref{lem:basic}.
\end{proof}

\subsection{Irreversible rank one elements of Coxeter groups}

There are obvious obstructions for a given element $\gamma \in W$ to have rank one: namely, if $\gamma$ is
contained in a parabolic subgroup $P < W$ which is of spherical or affine type, or which splits as a direct
product of two infinite subgroups, then clearly $\gamma$ cannot have rank one. The following shows that these
are in fact the only obstructions:

\begin{prop}\label{prop:Coxeter:charactRankOne}
An element $\gamma \in W$ does not have rank one if and only if $\gamma$ is contained in a parabolic subgroup $P
< W$ such that either $P$ is finite, or $P$ splits as $P = P_1 \times P_2$ where $P_1$ and $P_2$ are both
infinite parabolic subgroups, or $P$ splits as $P = K \times P_{\mathrm{aff}}$ where $K$ is a finite parabolic
subgroup and $P_{\mathrm{aff}}$ is an affine parabolic of rank $\geq 3$.
\end{prop}
\begin{proof}
The `if' part is clear; we focus on the `only if' part. Let thus $\gamma \in W$ be an element which does not
have rank one. If $\gamma$ is not hyperbolic, then $\gamma$ is of finite order and hence contained in a finite
parabolic subgroup as is well known. We may therefore assume that $\gamma$ is hyperbolic and the desired
assertion is provided by Proposition~\ref{prop:FlatHalfSpace}.
\end{proof}

\begin{cor}\label{cor:FullFlat}
Let $\gamma \in W$ be an element of infinite order and $L \subset \Sigma$ be an axis of $\gamma$. Then:
\begin{itemize}
\item[(i)] Either $\gamma$ is rank one or there exists $\gamma' \in W$ such that $\la \gamma, \gamma' \ra \cong
\ZZ \times \ZZ$.

\item[(ii)] $\gamma$ is rank one if and only if its centralizer is virtually cyclic.

\item[(iii)] $L$ is rank one if and only if it is not contained in a periodic $2$-flat.
\end{itemize}

\end{cor}
\begin{proof}
The first assertion follows from Proposition~\ref{prop:Coxeter:charactRankOne}. The second assertion follows
from the first and the  easy fact that rank one elements have virtually cyclic centralizer. Assertion (iii)
follows from (i) and the flat torus theorem \cite[Thm.~II.7.1]{Bridson-Haefliger}.
\end{proof}

\begin{cor}\label{cor:existence}
The group $W$ contains two elements $\gamma_1, \gamma_2$ such that $\gamma_1 \not \sim_W \gamma_2$ and $\gamma_1
\not \sim_W \gamma_2\inv$  if and only if $W$ is not a direct product of finite and affine Coxeter groups (or,
equivalently, if $W$ is not virtually abelian).
\end{cor}
\begin{proof}
Follows from Lemma~\ref{lem:irreversible}, Corollary~\ref{cor:EssentialElements} and
Proposition~\ref{prop:Coxeter:charactRankOne}, and the fact that for proper \cat spaces, a hyperbolic element is
contracting if and only if it is rank one, see \cite[Th.~5.4]{BestFuji_symsp}.
\end{proof}


Given a hyperbolic element $\gamma \in W$, we say that $\gamma$ is \textbf{irreversible} if $\gamma \not \sim_W
\gamma\inv$. In the case of Coxeter groups, there is a simple algebraic criterion which may be used to detect
irreversibility:

\begin{lem}\label{lem:Coxeter:irreversible}
A rank one element $\gamma \in W$ is irreversible if and only if no positive power of $\gamma$ can be written as
a product $\gamma^k = a.b$ where $a, b \in W$ have order~$2$.
\end{lem}

\begin{remark}
Lemma~\ref{lem:Coxeter:irreversible} can be used to obtain the following refinement of
Corollary~\ref{cor:EssentialElements}: if $W$ is infinite, irreducible and non-affine, then the Coxeter element
is irreversible as soon as the Coxeter diagram of $(W,S)$ is not a star, i.e. there is no element $s \in S$ such
that the parabolic subgroup $W_{S \minus \{s\}}$ is a finite elementary abelian $2$-group.
\end{remark}

\begin{proof}[Proof of Lemma~\ref{lem:Coxeter:irreversible}]
If $\gamma^k = a.b$ with $a, b$ involutions for some $k>0$, then $\gamma^k$ is conjugate to $\gamma^{-k}$ and,
hence, $\gamma\sim \gamma\inv$ by \cite[Prop.~6.5(3)]{BestFuji_symsp}. Thus $\gamma$ is not irreversible.

Suppose now that $\gamma$ is not irreversible. Then by properness  $W$ possesses an element $g$ which stabilises
some $\gamma$-axis $L$ and satisfies  $g\gamma g\inv|_L = \gamma\inv_L$. By Selberg's lemma $W$ possesses a
torsion free normal subgroup of finite index, which acts thus freely on $\Sigma$. Let $k > 0$ be such that
$\gamma^k$ belongs to this finite index subgroup. Then $g\gamma^k g\inv = \gamma^{-k}$. Since $W$ acts properly
on $\Sigma$, the subgroup of $W$ which stabilises $L$ is virtually cyclic; thus we may and shall assume that
$g^2$ acts trivially on $L$. In particular $g$ is a torsion element of $W$ of even order, say $2m$. Notice that
$g^m \gamma g^{-m}|_L = \gamma\inv_L$, whence $g^m \gamma^k g^{-m} = \gamma^{-k}$. We set $a = g^m =g^{-m}$ and
$b = g^m \gamma^k$. Then clearly $\gamma^k = a.b$ and $a^2 = 1 = b^2$, as desired.
\end{proof}


\section{Rank one isometries of buildings}\label{sect:buildings}\label{sec:bdg}

\subsection{Definitions and basic facts}

Let $(W,S)$ be a Coxeter system. A \textbf{building} of type $(W, S)$ is a set $\mathcal{C}$ endowed with a map
$\delta:  \mathcal{C} \times \mathcal{C} \to W$ submitted to the following conditions, where $x,y \in {\mathcal
C}$ and $w = \delta(x,y)$:

\begin{description}
\item[(Bu1)] $w = 1$ if and only if $x = y$;

\item[(Bu2)] if $z \in {\mathcal C}$ is such that $\delta(y,z) = s \in S$, then $\delta(x,z) = w$ or $ws$, and
if, furthermore, $l(ws) = l(w) + 1$, then $\delta(x,z) = ws$;

\item[(Bu3)] if $s \in S$, there exists $z \in {\mathcal C}$ such that $\delta(y,z) = s$ and $\delta(x,z) = ws$.
\end{description}

The map $\delta$ is called the \textbf{Weyl distance}. An \textbf{automorphism} of $(\mathcal{C}, \delta)$ is a
permutation of $\mathcal{C}$ which preserves the Weyl distance. The map $\delta_W : W \times W \to W$ defined by
$\delta_W(x, y) = x\inv y$ turns canonically $W$ into a building of type $(W,S)$. Any subset of a building
$(\mathcal{C}, \delta)$ of type $(W,S)$ which is Weyl-isometric to $(W, \delta_W)$ is called an
\textbf{apartment}. Given a subset $J \subseteq S$, we denote $W_J = \la J \ra$. Given any chamber $c \in
\mathcal C$, the set
$$ \Res_J(c) = \{x \in \mathcal C \; | \; \delta(c, x) \in W_J\}$$
is called the \textbf{residue} of type $J $ containing $c$. An important fact is that a residue of type $J$,
endowed with the appropriate restriction of the Weyl distance, is a building of type $(W_J, J)$. We refer to
\cite{Weiss_book} for the general theory.

\medskip An important fact is that any building of type $(W,S)$ possesses a geometric realization as a \cat
metric space \cite{Davis}. In other words, given a building $\mathscr{B} = (\mathcal{C}, \delta)$ of type
$(W,S)$, there exists a \cat space $X_\mathscr{B}$ and a canonical injection $\Aut(\mathscr{B}) \to
\Isom(X_\mathscr{B})$. We will identify all elements of $\Aut(\mathscr{B})$ to their image in
$\Isom(X_\mathscr{B})$.

\subsection{A characterization of rank one elements}

\begin{thm}\label{thm:CharRankOne}
Let $\mathscr{B} = (\mathcal{C}, \delta)$ be a building of type $(W,S)$ and let $\gamma \in \Aut(\mathscr{B})$
be a hyperbolic element. Then the following assertions are equivalent:
\begin{itemize}
\item[(i)] $\gamma$ is a rank one isometry of $X_\mathscr{B}$.

\item[(ii)] $\gamma$ is contracting.

\item[(iii)] $\gamma$ does not stabilise any residue whose Weyl group is of the form $W_I \times W_J$, where
either $W_I$ and $W_J$ are both infinite, or $W_I$ is affine and $W_J$ is finite.
\end{itemize}
\end{thm}

\begin{proof}
Let $L$ be an axis of $\gamma$ and $A$ be an apartment containing $L$; such an apartment exists by
\cite[Thm.~E]{CH06}. Let $\pi $ denote the nearest point projection to $L$ and $C \subset L$ be a compact
segment which is a fundamental domain for the $\la \gamma \ra$--action on $L$.

\medskip \noindent (i) $\Rightarrow$ (ii) Suppose for a contradiction that $\gamma$ is a rank one isometry but
that $L$ is not $B$-contracting for any $B$. Then there exist sequences $(x_n)$ and $(y_n)$ in $X_\mathscr{B}$
such that $d(\pi(x_n), \pi(y_n))$ tends to infinity with $n$. Upon applying appropriate elements, we may and
shall assume that $(x_n)$ is contained in $C$ for all $n$. Upon extracting a subsequence, we may further assume
that there exists a point $c \in L$, exterior to $C$, which separates $C$ from $\pi(y_n)$ for all $n$. We denote
by
$$\rho = \rho_{A, c}$$
the retraction onto $A$ centered at $c$. Recall that this map is the identity on $A$, it does not increase
distances and its restriction to any geodesic segment emanating from $c$ (and more generally to any apartment
containing $c$) is an isometry onto its image (see \cite[\S4.4]{AB06}). We claim that for any $z \in
X_\mathscr{B}$ such that $\pi(z) \neq c$, we have $d(c,\pi(\rho(z))) > d(c, \pi(z))$. Indeed, given such an $z$,
we have $\angle_{\pi(z)}(c, z) \geq \pi/2$. Therefore the nearest point projection of $c$ to the geodesic
segment $[z, \pi(z)]$ is $\pi(z)$. By the properties of the retraction $\rho$, this implies that the projection
of $c$ to $[\rho(z), \rho(\pi(z))]$ is $\rho(\pi(z)) = \pi(z)$. Thus $\angle_{\pi(z)}(c, \rho(z)) \geq \pi/2$
and it follows that $d(\pi(\rho(z)), c) \geq d(\pi(z), c)$, which proves the claim.

Applying this to $z=x_n$ and $z=x_n$, we deduce that
$$d(\pi(\rho(x_n)), \pi(\rho(y_n)))\geq d(\pi(x_n), \pi(y_n)),$$
which tends to $\infty$ with $n$. This shows that
the line $L$, viewed as a line in the apartment $A$, is not $B$-contracting for any $B$. On the other hand, the
map $\rho\circ \gamma|_A : A \to A$ is an isometry which has clearly rank one by assumption, since $A$ is a
closed convex subspace of $X_{\mathscr{B}}$. Thus it is $B$-contracting in view of
\cite[Thm.~5.4]{BestFuji_symsp}. This is a contradiction.

\medskip \noindent (ii) $\Rightarrow$ (iii) Assume that $\gamma$ stabilises a residue $R$ whose Weyl group is
of the form $W_I \times W_J$ as in (iii). Since $R$ is a building of type $(W_I \times W_J, I \cup J)$, it
follows that the Tits boundary of its \cat realisation $X_R$ has diameter $\pi$ and, hence, does not contain any
rank one isometry. In particular $X_R$ has no $B$-contracting isometry. The result follows, since $X_R$ is
isometrically embedded in $X_{\mathscr{B}}$.

\medskip \noindent (iii) $\Rightarrow$ (i) Suppose that $\gamma$ is not a rank one isometry;
in other words some $\gamma$-axis $L$ is contained in a flat half-plane, say $H$. The arguments of
\cite[Thm.~6.3]{CH06} show that $H$ is contained in an apartment $A$. Let $c \in L$ be any point and $\rho =
\rho_{A, c}$ be the retraction onto $A$ centered at $c$. Proposition~\ref{prop:FlatHalfSpace} implies that $\rho
\circ \gamma|_A$ is an isometry of $A$ contained in a parabolic subgroup of the form $W_I \times W_J$ as in the
statement of (iii). This implies that $\gamma$ stabilises the residue of type $I \cup J$ containing $c$, thereby
contradicting (iii).
\end{proof}

\subsection{Existence of rank one elements in Weyl--transitive groups} In order to deal with the question of
existence, we shall transfer to the whole building the constructions performed so far at the level of
apartments. An essential tool in doing this is the retraction that we have just considered.

\medskip
As before, let $\mathscr{B} = (\mathcal{C}, \delta)$ be a building of type $(W,S)$ and $g_1, g_2 \in
\Aut(\mathscr{B})$ be rank one elements. For $i \in \{1, 2\}$ let also $L_i$ be an axis of $g_i$, $A_i$ be an
apartment containing $L_i$, $c_i \in L_i$ be any point and $\rho_i = \rho_{A_i, c_i}$ be the retraction onto
$A_i$ centred at $c_i$. Then $\gamma_i := \rho_i \circ g_i|_{A_i} : A_i \to A_i$ is an automorphism of the
apartment $A_i$.

Recall that any apartment is isomorphic to the Davis complex $\Sigma$,  i.e. the standard \cat realization of the
thin building $(W, \delta_W)$. We now would like to compare $\gamma_1$ and $\gamma_2$ as elements of $W =
\Aut(W, \delta_W) < \Isom(\Sigma)$. In order to do this properly, we need to choose identifications $A_i \cong
\Sigma$ and make sure that our considerations are independent of this choice.

Crucial to us is the following:

\begin{lem}\label{lem:bdg:irreversible}
For $i \in \{1, 2\}$, let $f_i : A_i \cong \Sigma$ be any isomorphism (of thin buildings).

If the elements $g_1$ and $g_2$ are $\Aut(\mathscr{B})$-equivalent, then $f_1 \gamma_1 f_1\inv$ and $f_2
\gamma_2 f_2\inv$, viewed as elements of $W$, are $W$-equivalent.

If the elements $g_1$ and $g_2$ are not independent, then $f_1 \gamma_1 f_1\inv \sim_W f_2 \gamma_2 f_2\inv$ or
$f_1 \gamma_1 f_1\inv \sim_W f_2 \gamma_2\inv f_2\inv$.
\end{lem}

\begin{proof}
The first thing to observe is that any modification of the isomorphism  $f_1 : A_1 \cong \Sigma$ amounts to
replacing the element $f_1 \gamma_1 f_1\inv \in W$ by a $W$-conjugate. In view of Remark~\ref{rem:conjugacy},
the assertion of Lemma~\ref{lem:bdg:irreversible} is thus clearly independent of the choices of the $f_i$'s. In
order to avoid unnecessarily heavy notation, we shall henceforth identify both $A_1$ and $A_2$ to $\Sigma$ by
means of $f_1$ and $f_2$ respectively and, hence, omit to write the maps $f_1$ and $f_2$. In other words, the
elements $\gamma_1$ and $\gamma_2$ will be viewed as elements of $W$ acting on $\Sigma$.

\smallskip%
Fix a chamber $c_i \subset A_i$ such that $c_i$ meets $L_i$.
Upon replacing respectively $ g_1$ and  $g_2$  by some positive powers, we may and shall assume further that
\begin{equation}\label{eq:1}
\delta_W(c_i,  \gamma_i^n.c_i) = \delta_W(c_i, \gamma_i.c_i)^n
\end{equation}
for all $n > 0$ and $i = 1, 2$. Since furthermore the chambers $g_i^n.c_i$ and $\gamma_i^n.c_i$ intersect in a
point of $L_i$, the Weyl group element $\delta(g_i^n.c_i,\gamma_i^n.c_i)$ is contained in some standard finite
parabolic subgroup of $W$ for all $n>0$ and $i=1, 2$; in particular it is of uniformly bounded length. We deduce
that there exist an element $\vareps_{i, n} \in W$ of uniformly bounded length such that
\begin{equation}\label{eq:2}
\delta(c_i, g_i^n.c_i) = \delta_W(c_i, \gamma_i^n.c_i)\vareps_{i, n}
\end{equation}
for all $n>0$ and $i=1, 2$.

Suppose now that $g_1$ and $ g_2$ are $\Aut(\mathscr{B})$-equivalent. Then there exists a constant $D
>0$ and for each $n$ some element $g_n \in \Aut(\mathscr{B})$ such that
\begin{equation}\label{eq:3}
\ell \circ \delta(g_n.c_1, c_2) < D \hspace{.5cm} \text{and} \hspace{.5cm}
 \ell \circ \delta(g_n g_1^n. c_1, g_2^n c_2) < D
\end{equation}
for all $n > 0$, where $\ell : W \to \NN$ denotes the word length with respect to the Coxeter generating set
$S$. From~\eqref{eq:1},~\eqref{eq:2} and~\eqref{eq:3}, we deduce that there exist two sequences $(a_n)$, $(b_n)$
of elements of $W$, of uniformly bounded length, such that
$$
w_2^n = a_n.w_1^{n}.b_n
$$
for all $n>0$, where $w_i = \delta_W (c_i, \gamma_i.c_i) \in W$. Therefore, upon extracting a subsequence, we
obtain two elements $a, b \in W$ such that $w_2^{k_n} = a.w_1^{k_n}.b$ for all $n> 0$, where $(k_n)$ is a
strictly increasing sequence of positive integers.

Upon replacing $\gamma_1$ and $\gamma_2$ by a $W$-conjugate, we may assume that $c_1=c_2$ and that this chamber
corresponds to the identity element of $W$ (recall that the $1$-skeleton of $\Sigma$ is nothing but the Cayley
graph of $(W,S)$). This choice of parametrization yields $\gamma_1 = w_1$ and $\gamma_2=w_2$.

Let now $L_1^+, L_2^+ \in \bd \Sigma$ be the respective attracting fixed points of $w_1$, $w_2$ at infinity. Set
$c_0:= c_1 = c_2$. Since $w_2^n.c_0 \to L_2^+$ while $w_1^{n} b. c_0 \to L_1^+$ at the limit when $n$ tends to
infinity, if follows from the equality $w_2^{k_n} = a w_1 ^{k_n} b$ that $a.L_1^+ = L_2^+$. Thus $aw_1 a\inv$
and $w_2$ have the same attracting fixed point at infinity, namely $L_2^+$. By Lemma~\ref{lem:irreversible},
this implies that $w_1 \sim_W w_2$.

\medskip Assume that that $g_1$ and $g_2$ are not independent. In other words the axes $L_1$ and $L_2$
contain respectively rays which are asymptotic to each other. It follows that upon replacing $g_1$ and $g_2$ by
appropriate nonzero powers \eqref{eq:3} holds with $g_n \equiv 1$ for some $D >0$ and all $n \geq 0$. The same
argument as above may be repeated and now yields either $w_1 \sim_W w_2$ or $w_1 \sim_W w_2\inv$.
\end{proof}

\begin{prop}\label{prop:existence:building}
Let $\mathscr{B} = (\mathcal{C}, \delta)$ be a building of irreducible type $(W,S)$ and $G< \Aut(\mathscr{B})$
be a group of automorphisms acting Weyl--transitively on the chambers. Then $G$ contains two independent
elements $g_1, g_2$ such that $g_1 \not \sim_{\Aut(\mathscr{B})} g_2$ if and only if $(W, S)$ is neither
spherical nor affine (or, equivalently, if $W$ is not virtually abelian).
\end{prop}

\begin{proof}
The `only if' part is clear since, if $W$ is virtually abelian, then the Tits boundary of $X_{\mathscr{B}}$ is
either empty or of Tits diameter $\pi$. Suppose now that $W$ is not virtually abelian. Then, by
Corollary~\ref{cor:existence}, the group $W$ contains two rank one elements $\gamma_1, \gamma_2$ such that
$\gamma_1 \not \sim_W \gamma_2$ and $\gamma_1 \not \sim_W \gamma_2\inv$. Furthermore, the latter property
remains valid if we replace $\gamma_1$ and $\gamma_2$ by any nonzero power or any $W$-conjugate, see
Lemma~\ref{lem:irreversible} Remark~\ref{rem:conjugacy}. Therefore, we may and shall assume that some axis of
$\gamma_i$ $(i=1, 2)$ contains a point in the relative interior of a fixed base chamber $c$ of the \cat
realisation of $(W, \delta_W)$. Let $w_i = \delta_W(c, \gamma_i.c)$.

Fix now an apartment $A$ of $\mathscr{B}$, which identify it with $(W, \delta_W)$. In this way we view $c_0,
\gamma_i.c_0$ and $\gamma_i^2.c$ as chambers of $\mathscr{B}$, for $i=1, 2$. By hypothesis $G$ contains an
element $g_i$ such that $g_i.c=\gamma_i.c$ and $g_i^2.c = \gamma_i^2.c$. Since some $\gamma_i$-axis contains a
point in the relative interior of $c$, there exists a point $x_i$ in the relative interior of $\gamma_i.c$ such
that $\angle_{x_i}(\gamma_i\inv.x_i, \gamma_i.x_i)=\pi$. This implies that $\angle_{x_i}(g_i\inv.x_i,
g_i.x_i)=\pi$; in other words the points $\{g_i^n.x_i\}_{n \in \ZZ}$ are collinear, and hence $g_i$ is a
hyperbolic isometry, an axis of which contains $x_i$.

Let $A_i$ be an apartment containing $c$ and some axis $L_i$ of $\gamma_i$; such an apartment exists by
\cite[Th.~E]{CH06}. Let $\rho_i = \rho_{A_i, c}$ be the retraction onto $A$ centred at $c$. Then $\rho_i \circ
g_i$ is an automorphism of $A_i$ which maps $c$ to $g_i.c = \gamma_i.c$ and, hence, coincides with $\gamma_i$ if
we identify $A$ to $A_i$ by an means of isomorphism which fixes $c$. In particular, it follows from
Proposition~\ref{prop:Coxeter:charactRankOne} that the $g_i$-axis $L_i$ is not contained in any residue whose
Weyl group has the form $W_I \times W_J$ with $W_I$ and $W_J$ either both infinite or both virtually abelian. By
Theorem~\ref{thm:CharRankOne}, this implies that $g_i \in G$ is a rank one isometry of $X_{\mathscr{B}}$. Now
the fact that $g_1$ and $g_2$ are independent and $\Aut(\mathscr{B})$-inequivalent follows from
Lemma~\ref{lem:bdg:irreversible} in view of the definition of $\gamma_1$ and $\gamma_2$.
\end{proof}

We are now ready for the:
\begin{proof}[Proof of Theorem~\ref{thm:main}]
If $G$ is Weyl-transitive, then by Proposition~\ref{prop:existence:building}, the group $G$ contains two
independent rank one elements which are not $\Aut(\mathscr{B})$-equivalent.

If $\Stab_G(A)$ acts cocompactly on some apartment $A$, we may choose two elements of $\Stab_G(A)$ whose action on $A$ coincides
with some powers of the elements provided by Corollary~\ref{cor:existence}. These two elements of $\Stab_G(A)$ are rank
one for the same reason as in the proof of Proposition~\ref{prop:existence:building} above; they are independent
and $\Aut(\mathscr{B})$-inequivalent by Lemma~\ref{lem:bdg:irreversible}.

In view of Theorem~\ref{thm:CharRankOne}, we may apply Proposition~\ref{prop:infiniteHQH}, which yields the
desired conclusion.
\end{proof}

\begin{proof}[Proof of Corollary~\ref{cor:KM:CommutatorWidth}]
When $R$ is a field, the Kac--Moody group $\mathcal{G}(R)$ acts Weyl-transitively on each of its two buildings.
When $R$ is a domain, we consider the action of $\mathcal{G}(R)$ on either of the two buildings $\mathscr{B}_+$
and $\mathscr{B}_-$ associated with $\mathcal{G}(k)$, where $k$ is a field in which $R$ embeds. Since
$\mathcal{G}(R)$ already contains the Weyl group of $\mathcal{G}(k)$, it follows that $\mathcal{G}(R)$ acts
transitively on the chambers of the standard apartment of both  $\mathscr{B}_+$ and $\mathscr{B}_-$. In all
cases, the fact that $\QH(\mathcal{G}(R))$ is infinite-dimensional follows from Theorem~\ref{thm:main}.

The assertion on the stable commutator length now follows from \cite{Bavard}, while the assertion on the
commutator width follows from a straightforward verification.
\end{proof}

\begin{remark}
It follows in particular that a rank one element of the Weyl group of $\mathcal{G}(R)$ acts as a contracting
isometry on both $\mathscr{B}_+$ and $\mathscr{B}_-$.
\end{remark}

\begin{proof}[Proof of Corollary~\ref{cor:KM:simple}]
Immediate from Corollary~\ref{cor:KM:CommutatorWidth}, the simplicity result in \cite{CR06} and the fact that
Kac--Moody groups of different types over non-isomorphic finite fields are non-isomorphic \cite[Cor.~B]{CM06}.
\end{proof}

\subsection{A special case: buildings with isolated residues}

The aim of this section is to prove Proposition~\ref{prop:conj}. We first need an existence result for
hyperbolic isometries of proper Gromov hyperbolic metric spaces. It is certainly well known to the experts;
however we could not find a reference where it is explicitly stated in the literature. We therefore include  a
detailed proof.

\begin{prop}\label{prop:hyperbolic}
Let $X$ be a proper Gromov hyperbolic geodesic metric space and $G < \Isom(X)$ be any group of isometries. Then
exactly one of the following assertions holds:
\begin{itemize}
\item[(1)] $G$ contains a hyperbolic isometry.

\item[(2)] $G$ has a bounded orbit.

\item[(3)] $G$ has a unique fixed point at infinity.
\end{itemize}
\end{prop}

\begin{proof}
Let $\delta$ be a constant of hyperbolicity for the space $X$. We assume that $G$ does not contain any
hyperbolic isometry.

We start with the special case when $G$ is countable. We may then write $G$ as the union of an increasing chain
of finite subsets $S_1 \subset S_2 \subset \dots$. By \cite[Proposition~3.2]{Koubi}, for each $n$ the set $P_n$
consisting of those points $x \in X$ such that $d(g.x, x) \leq 100 \delta$ for all $g \in S_n$, is nonempty. If
each $P_n$ meets some fixed bounded subset of $X$, then $\bigcap_n P_n$ is nonempty since $X$ is proper and,
hence, $G$ has a bounded orbit. Otherwise, denoting by $\overline X$ the visual compactification $X \cup \bd X$,
the intersection $\bigcap_n \overline{P_n}$ is a subset of $\bd X$ which is pointwise fixed by $G$. If this
subset contains more than~$2$ points then $G$ has a bounded orbit; if it contains exactly two points then $G$
acts by translation along the geodesic lines joining them and, since $G$ has no hyperbolic element, we conclude
again that $G$ has a bounded orbit. Thus we are done in this special case.

We now turn to the general case and assume moreover that $G$ has no bounded orbit. In view of the above we may
assume that for every countable subgroup $H$ of $G$ the set $P_H$ of those points $x \in X$ such that $d(g.x, x)
\leq 100 \delta$ for all $g \in H$, is nonempty. The same arguments as before then yield the desired conclusion.
\end{proof}

\begin{proof}[Proof of Proposition~\ref{prop:conj}]
By \cite[Corollary~E]{Ca07} a building $X$ of type $(W, S)$ as in the statement possesses isolated Euclidean
residues. Thus it admits a realization as a proper Gromov hyperbolic geodesic metric space $|X|$ on which
$\Aut(X)$ acts by isometries, and such that the Euclidean residues correspond in a canonical way to the
parabolic points at infinity of $|X|$, see \cite{Bowditch}. The desired result now follows from
Proposition~\ref{prop:hyperbolic}.
\end{proof}

\begin{remark}
The results of \cite{Ca07} provide in fact a complete characterization of those buildings which are relatively
hyperbolic with respect to some family of (non-necessarily Euclidean) residues. The arguments above show that
Conjecture~\ref{conj} holds in that more general context. The remaining open case of buildings whose Weyl group
is \emph{not} relatively hyperbolic with respect to any family of finitely generated subgroups is especially
intriguing.
\end{remark}

\appendix
\section{On homogeneous quasi-morphisms of locally compact groups with integer values}

The purpose of this appendix is to prove the following.

\begin{thm}\label{thm:appendix}
Let $G$ be a locally compact group. Then any homogeneous quasi-morphism $\varphi: G \to \ZZ$ is continuous.
\end{thm}

It was observed by Roger Alperin that the solution to Hilbert fifth problem implies that any homomorphism of a
locally compact group to $\ZZ$ is continuous (this follows from \cite[Corollary~3]{Alperin}). The above
statement shows that this holds more generally for homogeneous quasi-morphisms. A remarkable result of a similar
nature has been established in \cite[Lemma~7.4]{BIW}, asserting that for any locally compact group $G$, a
homogeneous Borel quasi-morphism $G \to \RR$ is continuous. Notice that non-homogeneous quasi-morphisms are
generally discontinuous.

\medskip%
We start with a basic consequence of homogeneity.

\begin{lem}\label{lem:homogeneous}
Let $\varphi : G \to \RR$ be a homogeneous quasi-morphism of a group $G$, which vanishes on a normal subgroup
$N$. Then $\varphi$ descends to a homogeneous quasi-morphism of the quotient $G/N$.
\end{lem}

\begin{proof}
Given $g \in G$ and $n \in N$, we claim that $\varphi(g) = \varphi(g\cdot n)$. Indeed, for each integer $k>0 $,
there exists $n_k \in N$ such that $(g \cdot n)^k = g^k \cdot n_k$. Therefore, we have
$$\begin{array}{rcl}
\varphi(g\cdot n) &=& \lim_{k \to \infty}\frac{ \varphi\big( (gn)^k \big)}{k}\\
& = & \lim_{k \to \infty} \frac{\varphi\big( g^k \cdot n_k \big)}{k}\\
& \leq & \lim_{k \to \infty} \big( \frac{\varphi(g^k) + D}{k}\big)\\
& = & \varphi(g),
\end{array}$$
where $D$ is a constant depending only on $\varphi$. In particular, we have also $\varphi(g\cdot n\inv) \leq
\varphi(g)$, and hence $\varphi(gn) = \varphi(g)$. The desired conclusion follows.
\end{proof}

The next step is to consider totally disconnected groups, the key point being the compact case.

\begin{lem}\label{lem:profinite}
Let $G$ be a profinite group. Then any homogeneous quasi-morphism $\varphi: G \to \ZZ$ is constant.
\end{lem}

\begin{proof}
Assume for a contradiction that $\varphi$ is not constant and let $g \in G$ be such that $\varphi(g) \neq 0$.
Let $H$ be the closure of $\la g \ra$ in $G$. Thus $H$ is a pro-cyclic group. In particular it is Abelian and,
hence, amenable as an abstract group. It follows that the restriction of $\varphi$ to $H$ is a homomorphism.
Since $\ZZ$ is residually finite, the kernel of the restriction of $\varphi$ to $H$ is an intersection of finite
index subgroups of $H$. By \cite[\S4.2]{Serre_Galois}, any finite index subgroup of $H$ is closed. This shows
that the restriction of $\varphi$ to $H$ is continuous. Since $H$ is compact, we deduce $\varphi(H)=0$, a
contradiction.
\end{proof}

Recall from \cite[p.~55]{Kaplansky} (see also \cite[Satz~4]{HartmanRyll}) that a compact Abelian group is
connected if and only if it is divisible (in fact, the latter holds for non-Abelian groups as well, see
\cite[Corollary~2]{Mycielski}). From this and the preceding two lemmas, we deduce the following.

\begin{lem}\label{lem:compact}
Let $G$ be a compact group. Then any homogeneous quasi-morphism $\varphi: G \to \ZZ$ is constant.
\end{lem}

\begin{proof}
Let $G^\circ$ denote the neutral component of $G$. We claim that the restriction of $\varphi$ to $G^\circ$
vanishes. As in the proof of Lemma~\ref{lem:profinite}, it is enough to prove this fact in the case $G^\circ$ is
Abelian. As recalled above, a compact connected Abelian group is divisible. Thus $\varphi(G^\circ)=0$ since a
divisible group admits no nonzero homogeneous quasi-morphism. It follows from Lemma~\ref{lem:homogeneous} that
$\varphi$ descends to a quasi-morphism of the group of components $G/G^\circ$, and the desired conclusion now
follows from Lemma~\ref{lem:profinite}.
\end{proof}

The last and most important step is the following. It relies on the structure theory of locally compact connected
groups.

\begin{lem}\label{lem:one-param}
Let $G$ be a connected locally compact group. Then any homogeneous quasi-morphism $\varphi: G \to \ZZ$ is
constant.
\end{lem}

\begin{proof}
By \cite[Theorem~4.3]{Montgomery-Zippin}, the group $G$ possesses a compact normal subgroup $K$ such that $G/K$
is a Lie group. In view of Lemmas~\ref{lem:homogeneous} and~\ref{lem:compact}, there is no loss of generality in
assuming $K=1$. We suppose henceforth that $G$ is a Lie group. Let $R$ denote its soluble radical. Thus $R$ is a
connected Lie group which is amenable as an abstract group. In particular the restriction of $\varphi$ to $R$ is
a homomorphism, and we deduce $\varphi(R)=0$ since $R$ is generated by one-parameter subgroups. By
Lemma~\ref{lem:homogeneous}, we may thus further assume that $G$ is semi-simple. Appealing again to
Lemma~\ref{lem:compact}, it is enough to deal with the case when $G$ is simple and non-compact. Let $P$ be a
minimal parabolic subgroup of $G$. Then $P$ is soluble-by-compact and, hence, as before the restriction of
$\varphi$ to $P$ vanishes. The Bruhat decomposition now implies that every element of $G$ is a bounded product
of elements of a finite number of conjugates of $P$. Thus $\varphi(G)$ is bounded, whence constant by
homogeneity.
\end{proof}

\begin{proof}[Proof of Theorem~\ref{thm:appendix}]
Let $G^\circ$ denote the neutral component of $G$. By Lemma~\ref{lem:one-param}, the restriction of $\varphi$ to
$G^\circ$ vanishes. Thus Lemma~\ref{lem:homogeneous} ensures that $\varphi$ descends to a homogeneous
quasi-morphism of the group of components to $\ZZ$. In other words, it suffices to prove the theorem in the case
when $G$ is totally disconnected. It is enough to show that $\varphi\inv(0)$ is open. By~\cite[III \S4
No~6]{BourbakiTGI}, the group $G$ possesses some compact open subgroup $Q$. Thus $Q$ is a profinite group and
Lemma~\ref{lem:profinite} shows that $\varphi(Q)= 0$. Thus $\varphi\inv(0)$ is indeed open and we are done.
\end{proof}

{\small
\providecommand{\bysame}{\leavevmode\hbox to3em{\hrulefill}\thinspace}
\providecommand{\MR}{\relax\ifhmode\unskip\space\fi MR }
\providecommand{\MRhref}[2]{%
  \href{http://www.ams.org/mathscinet-getitem?mr=#1}{#2}
}
\providecommand{\href}[2]{#2}

}

\begin{thebibliography}{Ham08b}

\bibitem[AB08]{AB06}
Peter Abramenko and Kenneth~S. Brown, \emph{Buildings}, Graduate Texts in
  Mathematics, vol. 248, Springer, New York, 2008, Theory and applications.

\bibitem[Alp82]{Alperin}
Roger Alperin, \emph{Locally compact groups acting on trees}, Pacific J. Math.
  \textbf{100} (1982), no.~1, 23--32. 

\bibitem[Bal95]{BallmannLN}
Werner Ballmann, \emph{Lectures on spaces of nonpositive curvature}, DMV
  Seminar, vol.~25, Birkh\"auser Verlag, Basel, 1995, With an appendix by Misha
  Brin. 

\bibitem[Bav91]{Bavard}
Christophe Bavard, \emph{{Longueur stable des commutateurs}}, Enseign. Math.,
  IIe~Série \textbf{37} (1991), 109--150.

\bibitem[BB95]{BallmannBrin}
Werner Ballmann and Michael Brin, \emph{Orbihedra of nonpositive curvature},
  Inst. Hautes \'Etudes Sci. Publ. Math. (1995), no.~82, 169--209 (1996).

\bibitem[BFe]{BestvinaFeighn}
Mladen Bestvina and Mark Feighn,
\emph{A hyperbolic $Out(F_n)$-complex},
Preprint,  \texttt{arXiv:0808.3730}, 2008.


\bibitem[BF02]{BestvinaFujiwara:MCG}
Mladen Bestvina and Koji Fujiwara, \emph{Bounded cohomology of subgroups of
  mapping class groups}, Geom. Topol. \textbf{6} (2002), 69--89 (electronic).

\bibitem[BF07]{BestFuji_symsp}
M.~Bestvina and K.~Fujiwara, \emph{A characterization of higher rank symmetric
  spaces via bounded cohomology}, GAFA (to appear), preprint available on
  \texttt{arXiv:math/0702274v1}, 2007.

\bibitem[BH99]{Bridson-Haefliger}
Martin~R. Bridson and Andr\'e Haefliger, \emph{{Metric spaces of non-positive
  curvature}}, {Grundlehren der Mathematischen Wissenschaften 319, Springer,
  Berlin}, 1999.

\bibitem[BIW08]{BIW}
Marc Burger, Alessandra Iozzi, and Anna Wienhard, \emph{Surface group
  representations with maximal {T}oledo invariant}, Ann. Math. (to appear),
  2008.

\bibitem[BM96]{BM96}
Marc Burger and Shahar Mozes, \emph{{{\rm CAT($-1$)}}-spaces, divergence groups
  and their commensurators}, J. Amer. Math. Soc. \textbf{9} (1996), 57--93.

\bibitem[BM02]{BurgerMonod}
Marc Burger and Nicolas Monod, \emph{Continuous bounded cohomology and
  applications to rigidity theory}, Geom. Funct. Anal. \textbf{12} (2002),
  no.~2, 219--280. 

\bibitem[Bou71]{BourbakiTGI}
Nicolas Bourbaki, \emph{\'{E}l\'ements de math\'ematique. {T}opologie
  g\'en\'erale. {C}hapitres 1 \`a 4}, Hermann, Paris, 1971.

\bibitem[Bow99]{Bowditch}
Brian~H. Bowditch, \emph{Relatively hyperbolic groups}, Preprint, University of
  Southampton, 1999.


\bibitem[Cal]{Cal}
Danny Calegari, 
\emph{scl}, Monograph,
to appear from Japanese Math Society,
available at 
{\ttfamily http://www.its.caltech.edu/$\sim$dannyc/scl/toc.html}, 2009.


\bibitem[Cap07]{Ca07}
Pierre-Emmanuel Caprace, \emph{Buildings with isolated subspaces and relatively
  hyperbolic {C}oxeter groups}, Preprint, available at {\ttfamily
  arXiv:math.GR/0703799}, 2007.

\bibitem[Cap09]{Caprace:amenable}
\bysame, \emph{Amenable groups and {H}adamard spaces with a totally
  disconnected isometry group}, Comment. Math. Helv. \textbf{84} (2009),
  437--455.

\bibitem[CH06]{CH06}
Pierre-Emmanuel Caprace and Fr{\'e}d{\'e}ric Haglund, \emph{On geometric flats
  in the {${\rm CAT}(0)$} realization of {C}oxeter groups and {T}its
  buildings}, Canad. J. Math. (to appear), preprint available at {\ttfamily
  arXiv:math.GR/0607741}, 2006.

\bibitem[CM06]{CM06}
Pierre-Emmanuel Caprace and Bernhard M{\"u}hlherr, \emph{Isomorphisms of
  {K}ac-{M}oody groups which preserve bounded subgroups}, Adv. Math.
  \textbf{206} (2006), no.~1, 250--278. 

\bibitem[CM08]{CaMo}
Pierre-Emmanuel Caprace and Nicolas Monod, \emph{Isometry groups of
  non-positively curved spaces: structure theory}, Preprint,
  \texttt{arXiv:math.GR/0809.0457}, 2008.

\bibitem[CR09]{CR06}
Pierre-Emmanuel Caprace and Bertrand R\'emy, \emph{Simplicity and superrigidity
  of {K}ac-{M}oody lattices}, Invent. Math. \textbf{176} (2009), 169--221.




\bibitem[Dav98]{Davis}
Michael~W. Davis, \emph{Buildings are {${\rm CAT}(0)$}}, Geometry and
  cohomology in group theory, London Math. Soc. Lecture Note Ser., Cambridge
  Univ. Press, 1998, pp.~108--123.

\bibitem[EF97]{EpsteinFujiwara}
David B.~A. Epstein and Koji Fujiwara, \emph{The second bounded cohomology of
  word-hyperbolic groups}, Topology \textbf{36} (1997), no.~6, 1275--1289.

\bibitem[Ham08a]{Hamenstadt:CrossRatio}
Ursula Hamenst{\"a}dt, \emph{Bounded cohomology and cross ratios}, Preprint,
  \texttt{arXiv:math.GR/0508532}, to appear in ``Limits of graphs in group
  theory and computer science'', EPFL Lausanne Press, 2008.

\bibitem[Ham08b]{Hamenstadt_IsomHyperbolic}
\bysame, \emph{Isometry groups of proper hyperbolic spaces}, GAFA (to appear),
  2008.

\bibitem[Ham08c]{Hamenstadt}
\bysame, \emph{Subgroups of $Out(F_n)$},
Talk at the conference "Discrete groups and geometric structures" held in Kortrijk, May 2008.

\bibitem[HRN56]{HartmanRyll}
Stanis{\l}aw Hartman and Czes{\l}aw Ryll-Nardzewski, \emph{Zur {T}heorie der
  lokal-kompakten {A}belschen {G}ruppen}, Colloq. Math. \textbf{4} (1956),
  157--188. 

\bibitem[Kap54]{Kaplansky}
Irving Kaplansky, \emph{Infinite abelian groups}, University of Michigan Press,
  Ann Arbor, 1954. 

\bibitem[Kou98]{Koubi}
Malik Koubi, \emph{Croissance uniforme dans les groupes hyperboliques}, Ann.
  Inst. Fourier (Grenoble) \textbf{48} (1998), no.~5, 1441--1453. 

\bibitem[Kra09]{Krammer}
Daan Krammer, \emph{The conjugacy problem for {C}oxeter groups}, Groups Geom.
  Dyn. \textbf{3} (2009), 71--171.

\bibitem[MK99]{Kourovka}
V.~D. Mazurov and E.~I. Khukhro (eds.), \emph{The {K}ourovka notebook},
  augmented ed., Russian Academy of Sciences Siberian Division Institute of
  Mathematics, Novosibirsk, 1999, Unsolved problems in group theory.

\bibitem[Mou88]{Moussong}
G.~Moussong, \emph{Hyperbolic coxeter groups}, Ph.D. Thesis, Ohio State
  University, 1988.

\bibitem[MR06]{MonodRemy}
Nicolas Monod and Bertrand R\'emy, \emph{Boundedly generated groups with
  pseudocharacter(s)}, J. London Math. Soc. (2) \textbf{73} (2006), no.~1,
  84--108, Appendix to {\it {Q}uasi-actions on trees and Property ({QFA})} by
  J.~F.~Manning. 

\bibitem[Mur07]{Muranov}
Alexey Muranov, \emph{Finitely generated infinite simple groups of infinite
  commutator width}, Internat. J. Algebra Comput. \textbf{17} (2007), no.~3,
  607--659. 

\bibitem[Myc57]{Mycielski}
Jan Mycielski, \emph{Some properties of connected compact groups}, Colloq.
  Math. \textbf{5} (1957), 162--166. 

\bibitem[MZ55]{Montgomery-Zippin}
Deane Montgomery and Leo Zippin, \emph{Topological transformation groups},
  Interscience Publishers, New York-London, 1955.

\bibitem[Par07]{Paris_irrCox}
Luis Paris, \emph{Irreducible {C}oxeter groups}, Internat. J. Algebra Comput.
  \textbf{17} (2007), 427--447.

\bibitem[Pra77]{Prasad77}
Gopal Prasad, \emph{Strong approximation for semi-simple groups over function
  fields}, Ann. of Math. (2) \textbf{105} (1977), no.~3, 553--572.

\bibitem[R{\'e}m02]{RemAst}
Bertrand R{\'e}my, \emph{Groupes de {K}ac-{M}oody d\'eploy\'es et presque
  d\'eploy\'es}, Ast\'erisque \textbf{277} (2002), viii+348 pages.

\bibitem[Ser94]{Serre_Galois}
Jean-Pierre Serre, \emph{Cohomologie galoisienne}, fifth ed., Lecture Notes in
  Mathematics, vol.~5, Springer-Verlag, Berlin, 1994. 

\bibitem[Sha00]{Shalom2000}
Yehuda Shalom, \emph{Rigidity of commensurators and irreducible lattices},
  Invent. Math. \textbf{141} (2000), 1--54.

\bibitem[Tit69]{Tits_CoxWordPb}
Jacques Tits, \emph{Le probl\`eme des mots dans les groupes de {C}oxeter},
  Symposia Mathematica (INDAM, Rome, 1967/68), Vol. 1, Academic Press, London,
  1969, pp.~175--185. 

\bibitem[Tit87]{TitsJA}
\bysame, \emph{Uniqueness and presentation of {K}ac-{M}oody groups over
  fields}, J. Algebra \textbf{105} (1987), 542--573.

\bibitem[Tit92]{TitsLMS}
\bysame, \emph{Twin buildings and groups of {K}ac-{M}oody type}, in Lecture
  Notes of the London Mathematical Society (Durham, 2000), London Mathematical
  Society, 1992, pp.~249--286.

\bibitem[Wei03]{Weiss_book}
Richard~M. Weiss, \emph{The structure of spherical buildings}, Princeton
  University Press, Princeton, NJ, 2003. 

\end{thebibliography}
\end{document}